\documentclass[11pt,oneside]{amsart}

\usepackage{amssymb, amsmath, amsthm}
\usepackage{mathabx}
\usepackage{graphicx}
\usepackage[utf8]{inputenc}
\usepackage{hyperref}
\usepackage{mathtools}
\usepackage{rotating}
\usepackage{tabularx}
\usepackage{dynkin-diagrams}
\usepackage{subcaption} 
\usepackage[all]{xy}
\usepackage{stmaryrd}

\usepackage[left=2.5cm,right=2cm,top=2cm,bottom=2cm]{geometry}

\newtheorem{thm}{Theorem}
\newtheorem*{thm*}{Theorem}
\newtheorem{cor}[thm]{Corollary}
\newtheorem{lem}[thm]{Lemma}
\newtheorem{prop}[thm]{Proposition}

\newtheorem{rem}[thm]{Remark}

\theoremstyle{definition}

\numberwithin{thm}{section}
\numberwithin{table}{section}

\newcommand{\Z}{\mathbb{Z}}
\newcommand{\N}{\mathbb{N}}
\newcommand{\weyl}{\mathcal{W}}
\newcommand{\schub}{\mathcal{S}}
\newcommand{\g}{\mathfrak{g}}
\newcommand{\code}{\mathsf{code}}
\newcommand{\raiz}{\delta}
\newcommand{\CK}[2]{\mathop{\langle #1^{\vee}, #2 \rangle}}
\newcommand{\CKalpha}[2]{\CK{\raiz_{#1}}{\raiz_{#2}}}
\newcommand{\CKdual}[2]{\mathop{\langle #1^{\vee\vee}, #2^{\vee} \rangle}}
\newcommand{\CKalphadual}[2]{\CKdual{\raiz_{#1}}{\raiz_{#2}}}
\newcommand{\height}{\mathrm{ht}}

\newcommand{\w}[1]{\llbracket #1\rrbracket}
\newcommand{\I}{I}
\renewcommand{\hat}{\widehat}

\title[A correspondence between coefficients of flag manifolds and height of roots]{A correspondence between boundary coefficients of real flag manifolds and height of roots}

\author[lambert]{Jordan Lambert}
\address[Jordan Lambert]{Department of Mathematics -- ICEx, Universidade Federal Fluminense, Volta Redonda 27213-145, Rio de Janeiro, Brazil}\email{jordanlambert@id.uff.br}

\author[rabelo]{Lonardo Rabelo}
\address[Lonardo Rabelo]{Department of Mathematics, Federal University of Juiz de Fora, Juiz de Fora 36036-900, Minas Gerais, Brazil}
\email{lonardo@ice.ufjf.br}

\thanks{This work was supported by the Coordination for the Improvement of Higher Level Personnel -- Capes}



\begin{document}

\begin{abstract}
In this paper we prove a new formula for the coefficients of the cellular homology of real flag manifolds in terms of the height of certain roots. In particular, for flag manifolds of type A, we get a very simple formula for these coefficients and an explicit expression for the first and second homology groups with integer coefficients.
\end{abstract}

\maketitle

\allowdisplaybreaks %% Quebra equações matematicas no Align

\section{Introduction}

In this work we revisit the determination of the coefficient of the boundary map for the cellular homology of real flag manifolds, a problem which is equivalent to that of finding the incidence coefficients of the differential map for the cohomology. It follows a series of previous work on this subject like Kocherlakota \cite{Koc95} with its Morse Theory approach, Casian-Stanton \cite{CS99} by a representation theoretical view, Rabelo-San Martin \cite{RSm19} within the frame of the  cellular homology and, lastly, Matszangosz \cite{Mat19} throughout the cohomology of the Vassiliev complex. We prove a new formula for these coefficients in terms of the height of certain roots.

A generalized flag manifold $\mathbb{F}$ is a homogeneous space $G/P$, where $G$ is a real noncompact semisimple Lie Group and $P$ is a parabolic subgroup. It admits a cellular decomposition called Bruhat decomposition, where the cells are the Schubert cells and parametrized by the Weyl group $\mathcal{W}$. Consider a pair of Schubert cells $\schub_w,\schub_{w'}$ such that $w$ covers $w'$, i.e., $w'\leq w$ by the Bruhat-Chevalley order and $\ell(w')=\ell(w)-1$. In this case, there is a root $\beta$ such that $w=s_{\beta}\cdot w'$. According to both \cite{Koc95} and \cite{RSm19}, we may summarize how to compute the coefficient $c(w,w')$ as it follows: Consider the set $\Pi_w$ of positive roots sent to negative by $w^{-1}$ and denote by $\phi(w)$ the sum of all such roots. Hence, $c(w,w')=\pm(1+(-1)^{\kappa(w,w')})$, where $\kappa$ is obtained from the equation $\phi(w)-\phi(w')=\kappa(w,w')\cdot \beta$. The papers \cite{Rab16} and \cite{LR19} apply this procedure in the context of the Isotropic Grassmannians. The results obtained in these works (for instance, see \cite{LR19a}, Theorem 3.12) suggest a formula for the coefficients in terms of the height of some root. This is our first Theorem, which is remarkable by its simplicity. 

\begin{thm*} If we write $w=w'\cdot s_{\gamma}$, it follows that $\kappa(w,w')=\mathrm{ht}(\gamma^{\vee})$, where $\mathrm{ht}(\gamma^{\vee})$ is the height of the corresponding dual coroot $\gamma^{\vee}$ with respect to the dual system of roots.
\end{thm*}

In other words, the parity of $\mathrm{ht}(\gamma^{\vee})$ determines whether the coefficient $c(w,w')$ is zero or $\pm 2$. Notice the role developed by the roots $\beta$ and $\gamma$ or, equivalently, by the right and left actions in the understanding of this topic.

In the context of flags of type $A_{n-1}$, this theorem simplifies a lot the task of computing the coefficient. Schubert cells are parametrized by the symmetric group $S_n$. If we denote a permutation by the one-line notation as $w=w_1\cdots w_n$ then $w$ covers $w'$ if and only if there is $i<j$ such that $w=w'\cdot (i,j)$, $w_i'<w_j'$ and there is no $i<k<j$ such that $w'_i<w_k<w'_j$. Since the simple system of roots is equal to its dual,
it is immediate to conclude that $\kappa(w,w')=j-i$ (see Proposition \ref{prop:typeAcoef}).

It provides a very nice link between combinatorics and topology. As a consequence, we retrieve the orientability condition for this class of flags according to Patr\~ao-San Martin-Santos-Seco \cite{PSS12} (Proposition \ref{prop:orient}).

We may also go further and derive an explicit formula for 1, 2-homology of any partial flag manifold of type A.
\begin{thm*}
Given any partial flag manifold $\mathbb{F}_{\Theta}$ of type $A$:
\begin{enumerate}
\item For $n\geqslant 3$, the 1-homology is given by
\begin{equation*}
H_{1}(\mathbb{F}_{\Theta},\Z) \cong (\Z_{2})^{n-|\Theta|-1}.
\end{equation*}

\item For $n\geqslant 4$, the 2-homology is given by
\begin{equation*}
H_{2}(\mathbb{F}_{\Theta},\Z) \cong (\Z_{2})^{{n-|\Theta|-1 \choose 2} +r_{\Theta} - 1},
\end{equation*}
where $r_{\Theta}$ is the number of connected components of the Dynkin diagram of $\Theta$.
\end{enumerate}
\end{thm*}

With respect to the 1-homology group, as the abelianization of the fundamental group, it may be also derived from the work of Wiggerman \cite{Wig98} which gives a presentation of $\Pi_1(\mathbb{F}_{\Theta})$ with generators in the set $\Sigma \setminus \Theta$ subject to some relations. In this sense, although it is not a kind of new result, our combinatorial approach presents a direct and simple computation for such abelianization. A similar result that should be mentioned is that in the context of the 2-homotopy of the complex flag manifolds, Grama-Seco (\cite{GS20}, Theorem 2.4) obtained by geometric methods that it is generated by $2$-spheres which are also enumerated by the complement of $\Theta$ inside $\Sigma$. With respect to the 2-homology group, Del Barco-San Martin (\cite{BS19}, Theorem 4.1) has shown that it is a torsion group $\mathbb{Z}_{2}^{N}$ without providing any description for $N$.

The article is organized as follows.
In the Section \ref{sec:preliminaries}, we introduce the main definitions about flag manifolds, root and coroot systems, Bruhat decomposition, and combinatorics of the symmetric group. 
In the next Section \ref{sec:height}, we prove the height formula for the coefficients and derive some consequences. 
Finally, in the Section  \ref{sec:conclusion} we point some further directions out. 

\section{Preliminaries}\label{sec:preliminaries}

Let $\N=\{1,2,3, \dots\}$ and $\Z$ be the set of integers. For $n,m\in\Z$, where $n\leqslant m$, denote the set $[n,m]=\{n,n+1, \dots, m\}$. For $n\in \N$, denote $[n]=[1,n]$.

\subsection{Flag Manifolds}

We will begin by defining all required structure to deal with flag manifolds. 

We define the flag manifolds as homogeneous spaces $G/P$ where $G$ is a non-compact semi-simple Lie group and $P$ is a parabolic subgroup of $G$. The flag manifolds for the several groups $G$ with non-compact real semi-simple Lie algebra $\mathfrak{g}$ are the same. 

If $\mathfrak{g} = \mathfrak{k} \oplus \mathfrak{s}$ is a Cartan decomposition, let $\mathfrak{a}$ be a maximal abelian sub-algebra contained in $\mathfrak{s}$. A sub-algebra $\mathfrak{h} \subset \mathfrak{g}$ is said to be a Cartan sub-algebra if $\mathfrak{h}_{\mathbb{C}}$ is a Cartan sub-algebra of $\mathfrak{g}_{\mathbb{C}}$. If $\mathfrak{h} = \mathfrak{a}$ is a Cartan sub-algebra of $\mathfrak{g}$, we say that $\mathfrak{g}$ is a split real form of $\mathfrak{g}_{\mathbb{C}}$. 

Let $\Pi$ be the set of roots of the pair $(\mathfrak{g}, \mathfrak{a})$ and 
fix a simple system of roots $\Sigma \subset \Pi$. Denote by $\Pi^{\pm}$ respectively the set of positive and negative
roots and by $\mathfrak{a}^+$ the Weyl chamber $\mathfrak{a}^+ = \{ H \in \mathfrak{a} \colon \alpha (H) > 0 \mbox{ for all }
\alpha \in \Sigma \}$. The direct sum of root spaces corresponding to the positive roots is denoted by $\mathfrak{n} = \sum_{\alpha \in \Pi^+} \mathfrak{g}_\alpha$. The Iwasawa decomposition of $\mathfrak{g}$ is given by $\mathfrak{g} = \mathfrak{k} \oplus \mathfrak{a} \oplus \mathfrak{n}$. The notations $K$ and $N$ refer respectively to the connected subgroups whose Lie algebras are $\mathfrak{k}$ and $\mathfrak{n}$.

A minimal parabolic sub-algebra of $\mathfrak{g}$ is given by $\mathfrak{p} = \mathfrak{m} \oplus \mathfrak{a} \oplus \mathfrak{n}$ where $\mathfrak{m}$ is the centralizer of $\mathfrak{a}$ in $\mathfrak{k}$. Let $P$ be the
minimal parabolic subgroup with Lie algebra $\mathfrak{p}$. Note that $P$ is the normalizer of $\mathfrak{p}$ in $G$.
We call $\mathbb{F} = G/P$ the maximal flag manifold of $G$ and denote by $b_0$ the base point $1 \cdot P$ in $G/P$.

Now, assume that $\Theta\subset \Sigma$ is any subset of simple roots. Such a choice provides a very interesting way to obtain several flag manifolds, called partial flag manifolds as we now explain. Denote by $\mathfrak{g}(\Theta)$ the semi-simple Lie algebra generated by $\mathfrak{g}_{\pm \alpha}$, $\alpha \in \Theta$. Let $G(\Theta)$ be the connected group with Lie algebra $\mathfrak{g}(\Theta)$. Yet, let $\mathfrak{n}_\Theta$ be the sub-algebra generated by the roots spaces $\mathfrak{g}_{-\alpha}$, $\alpha \in \Theta$ and consider $\mathfrak{p}_\Theta = \mathfrak{n}_\Theta \oplus \mathfrak{p}$. The normalizer $P_\Theta$ of $\mathfrak{p}_\Theta$ in $G$ is a standard parabolic subgroup which contains $P$. Finally, the corresponding flag manifold $\mathbb{F}_\Theta = G/P_{\Theta}$ is called a partial flag manifold of $G$ of type $\Theta$. We denote by $b_\Theta$ the base point $1\cdot P_\Theta$ in $G/P_\Theta$.

\subsection{Root Systems and Coroots}
In this section, we highlight some results about the dual system of a root system which will play a key role in the sequel. We follow closely the book of Perrin \cite{Per15}.  

Let $E$ be a finite dimensional vector space. A set $\Pi \subset E$ is an abstract system of roots if it is finite, spans $E$ and does not contain $0$ and satisfies
\begin{enumerate}
\item for every $\alpha \in \Pi$, there exists a reflection $s_{\alpha}$ with respect to $\alpha$ such that $s_{\alpha}(\Pi)=\Pi$;
\item for every $\alpha,\beta \in \Pi$, $s_{\alpha}(\beta)-\beta$ is an integer multiple of $\alpha$. 
\end{enumerate}

The Weyl group $\weyl$ of the root system $\Pi$ is the group generated by reflections $s_{\alpha}$, $\alpha \in \Pi$. It is possible to show that there exists an inner product $\langle \cdot, \cdot \rangle$ in $E$ which is invariant by $\weyl$. It follows that $s_{\alpha}$ is the corresponding orthogonal reflection 
\begin{equation*}
s_{\alpha}(\beta) = \beta - 2\frac{\langle \beta, \alpha\rangle}{\langle \alpha, \alpha \rangle} \alpha.
\end{equation*}

From the identification of $E$ with $E^{\ast}$ by $\langle \cdot, \cdot \rangle$, for each root $\alpha \in \Pi$, let us denote by $\alpha^{\vee}=\dfrac{2\alpha}{\langle \alpha,\alpha\rangle}$. We call $\alpha^{\vee}$ the coroot of $\alpha$. It follows that if $\beta\in \Pi$ is another root, 
\begin{equation}\label{eq:acao_raizes}
s_{\alpha}(\beta) = \beta - \CK{\alpha}{\beta}\alpha.
\end{equation}

We know that if we choose $\langle \cdot, \cdot \rangle$ to be the Killing form, then the set of roots $\Pi$ associated to the Lie algebra $\mathfrak{g}$ is an abstract root system. 
%A root system is called reduced when for $\alpha\in \Pi$, $\mathbb{R}\alpha \cap \Pi =\{-\alpha,\alpha\}$. Denote by $\Pi^{\ast}$ the set of coroots of $\Pi$. 
Furthermore, the set of coroots $\Pi^{\ast}$ is also a root system in $E^{\ast}$ which is called the dual root system.

\begin{prop}[\cite{Per15}, Proposition 11.4.1]
Let $\alpha, \beta \in \Pi$. Then, $(s_{\alpha}\beta)^{\vee} = s_{\alpha^{\vee}}(\beta^{\vee})$.
\end{prop}

Hence, we derive the following corollary that will be very useful.

\begin{cor}\label{cor:dualroot}
Suppose that $\alpha\in \Pi$ is given by $\alpha = s_{1}\cdots s_{m-1}(\raiz_{m})$ such that $s_{i} = s_{\raiz_{i}}$ is the simple reflection associated to $\raiz_{i}\in \Pi$. Then, the coroot $\alpha^{\vee}$ can be written as follows
\begin{equation}
\alpha^{\vee} = s_{\raiz_{1}^\vee}\cdots s_{\raiz_{m-1}^\vee}(\raiz_{m}^{\vee}).
\end{equation}
\end{cor}
 
A system of simple roots $\Sigma \subset \Pi$ is a basis of $E$ such that every root $\alpha \in \Pi$ is written as a linear combination with integer coefficients with the same signal, i.e., all of them either non-negative or non-positive. 
A root system is called reduced when for $\alpha\in \Pi$, $\mathbb{R}\alpha \cap \Pi =\{-\alpha,\alpha\}$. Such a reduced root system is always the set of roots of a Cartan subalgebra over an algebraically closed field. Non-reduced root systems appear in real semi-simple Lie algebras. However, split real forms correspond to reduced root systems.

\begin{prop} [\cite{Per15}, Prop. 11.6.13] If $\Pi$ is reduced then $\Sigma^{\ast}=\{\alpha^{\vee}\colon \alpha \in \Sigma\}$ is a simple root system of $\Pi^\ast$. 
\end{prop}

By this Proposition, if $\alpha \in \Pi$ is given in terms of the system $\Sigma$ of simple roots as
\begin{equation}\label{eq:altura}
\alpha = \sum_{\delta \in \Sigma} d_{\delta} \delta
\end{equation}
we should obtain an analogous expression for $\alpha^{\vee}$ with respect to the simple root system $\Sigma^{\ast}$. Indeed, it follows that
\begin{equation}\label{eq:coroot_beta}
\alpha^{\vee} = \sum_{\delta \in \Sigma} d_{\delta} \frac{\langle \delta,\delta\rangle}{\langle \alpha,\alpha \rangle} \delta^{\vee}
\end{equation}

As a consequence, if $\alpha^{\vee}=\sum_{\delta^{\vee} \in \Sigma^{\ast}} d_{\delta}^{\ast} \delta^{\vee}$ is the coroot of $\alpha$ given by \eqref{eq:altura}, the relationship between its coefficients is given by
\begin{equation*}
d_{\delta}^{\ast} = d_{\delta} \frac{\langle \delta,\delta\rangle}{\langle \alpha,\alpha \rangle}.
\end{equation*}

By Equation \ref{eq:coroot_beta}, we also conclude that if $\mathfrak{g}$ is of type $A,D,E$ then its dual root system is isomorphic to itself while the dual root system of a lie algebra of type B is isomorphic to the root system of type C (and vice-versa).

The height of the root $\alpha$, denoted by $\mathrm{ht}(\alpha)$ is the sum of the coefficients that appear in the decomposition of $\alpha$ in Equation \eqref{eq:altura}
\begin{equation*}
\mathrm{ht}(\alpha)=\sum_{\delta \in \Sigma} d_{\delta}.
\end{equation*}

\subsection{Bruhat decomposition}

If we consider the elements of $\mathcal{W}$ as product of simple reflections $s_i=s_{\alpha_i}$, $\alpha_i \in \Sigma$, it is defined the length $\ell(w)$ of $w \in \mathcal{W}$ as the number of simple reflections in any reduced decomposition of $w$. 

There is a partial order $\leq$ in the Weyl group called the Bruhat-Chevalley order: we say that $w_{1}\leqslant w_{2}$ if given a reduced decomposition $w_2 = s_{j_{1}} \cdots s_{j_{r}}$ then $w_{1}=s_{j_{i_{1}}}\cdots s_{j_{i_{k}}}$ for some $1\leqslant i_1\leqslant \cdots \leqslant i_r\leqslant r$. When there exists $w,w' \in \weyl$ such that $w'\leqslant w$ and $\ell(w) = \ell(w')+1$ we say that $w$ covers $w'$ (alternatively, $w,w'$ is a covering pair). If $w$ covers $w'$ and given a reduced decomposition $w=s_{1}\cdots s_{\ell}$ then we will denote by $\I $ the integer in $[\ell]$ such that $w' = s_{1} \cdots \widehat{s_{\I}}\cdots s_{\ell}$, where the integer $\I$ depends on $w'$ and the choice of the reduced decomposition of $w$. For convenience, we will sometimes refer to this decomposition of $w'$ as $\widehat{w}_\I$.

For the subset $\Theta \subset \Sigma$, we define the subgroup $\mathcal{W}_\Theta $ generated by the
reflections with respect to the roots $\alpha \in \Theta$. 
We denote by $\mathcal{W}^{\Theta}$ the subset of minimal representatives of the cosets of $\mathcal{W}_{\Theta}$ in $\mathcal{W}$.

The Bruhat decomposition presents flag manifolds as union of $N$-orbits, namely,
\begin{equation*}
\mathbb{F}_{\Theta }=\coprod_{w\in \mathcal{W}^{\Theta }} N\cdot wb_{\Theta }.
\end{equation*}

Each orbit $N\cdot wb_{\Theta}, w \in \mathcal{W}$, is called a Bruhat cell. It is diffeomorphic to a euclidean space and, in the case of a split real form, its dimension coincides with the length of $w$, i.e., $\dim \left( N\cdot wb_{\Theta }\right) = \ell(w).$ A Schubert variety $\mathcal{S}_{w}$ is the closure of a Bruhat cell. The Bruhat-Chevalley order also characterizes a partial order between the corresponding Schubert varieties. It also endows the flag manifolds with a cellular structure where $\mathcal{S}_{w}= {\bigcup_{u\leq w} N\cdot ub_{\Theta}}$.

\subsection{Recursive formula for the roots of \texorpdfstring{$\Pi_w$}{PIw}}
A very important role is developed by the set $\Pi_w = \Pi^+ \cap w \Pi^-$ composed of the positive roots sent to negative roots by $w^{-1}$.  If $w = s_1 \cdots s_m$ is a reduced decomposition of $w$ then $\Pi_{w}=\{\beta_1, \ldots, \beta_m\}$ where
\begin{align}\label{eq:betat}
\beta_{k} = s_1 \cdots s_{k-1}(\raiz_{k})~,~ \mbox{for} ~ 1\leq k \leq m.
\end{align} 
In particular, we have that $\ell(w)$ equals the cardinality of $\Pi_w$. In this section, we consider a slightly more general setting in which we derive a recursive formula for a root that is obtained after a finite sequence of composition of reflections over a simple root. As a consequence, we obtain a formula for those roots of $\Pi_{w}$.

Consider a finite ordered sequence of simple roots $(\raiz_{1}, \dots, \raiz_{m})$, eventually with repetition. 
We are interested in writing $s_1 \cdots s_{m-1}(\raiz_{m})$ in terms of $\raiz_{1}, \dots, \raiz_{m}$. 
When $m=1$, it is trivial. However, as $m$ increases, we may observe the occurrence of a pattern. Let us show what happens for $m=2,3,4$ after successive applications of Equation \eqref{eq:acao_raizes}.
\begin{align*}
s_{1}(\raiz_{2}) & = \raiz_{2} - \CKalpha{1}{2} \raiz_{1} \\
s_{1}s_{2}(\raiz_{3}) &= \raiz_{3} - \CKalpha{2}{3}\raiz_{2} + \left( - \CKalpha{1}{3}  + \CKalpha{1}{2} \CKalpha{2}{3}\right)\raiz_{1} \\
s_{1}s_{2}s_{3}(\raiz_{4}) &= \raiz_{4} - \CKalpha{3}{4} \raiz_{3} +  \left( - \CKalpha{2}{4}  + \CKalpha{2}{3} \CKalpha{3}{4} \right)\raiz_{2} + \\
 + & \left( -\CKalpha{1}{4} + \CKalpha{1}{2}\CKalpha{2}{4} + \CKalpha{1}{3}\CKalpha{3}{4}  - \CKalpha{1}{2} \CKalpha{2}{3} \CKalpha{3}{4} \right) \raiz_1.
\end{align*}

By these equations, it turns out that $s_1 \cdots s_{m-1}(\raiz_{m})$ is a combination of some integer coefficients in terms of the roots $\raiz_{1}, \dots, \raiz_{m}$. The coefficients are given as alternating sums of products of the Killing numbers of the roots that belong to a specific interval. Besides, as $m$ increases, the sum is given by products of a greater number of factors. We now proceed to obtain a general formula for this expression. 

Let us define a sum formula in terms of the ordered sequence of roots $(\raiz_{1}, \dots, \raiz_{m})$: given integers $x,y$ such that $1 \leqslant x<y \leqslant m$ and  $0\leqslant l < y-x$, we define
\begin{equation}
P^{l}_{x,y}(\raiz_{1}, \dots, \raiz_{m}) = \sum_{x<j_1<\cdots<j_l<y} \CKalpha{x}{j_{1}} \CKalpha{j_{1}}{j_{2}} \cdots  \CKalpha{j_{l-1}}{j_{l}}\CKalpha{j_{l}}{y},
\end{equation}
which the sum runs among all choices of $l$-uples $x<j_1<\cdots<j_l<y$ of the following product $\CKalpha{x}{j_{1}} \CKalpha{j_{1}}{j_{2}} \cdots \CKalpha{j_{l}}{y}$. In particular, for $l=0$ we define
\begin{equation}
P_{x,y}^{0}(\raiz_{1}, \dots, \raiz_{m}) = \CKalpha{x}{y} 
\end{equation}
as the Killing number of the $x$-th coroot with the $y$-th root of the sequence $(\raiz_{1}, \dots, \raiz_{m})$.

Notice that this definition depends on the order of the roots and the integers $x$ and $y$ corresponds respectively to the position in the sequence that gives the first coroot and the last root of the factors of the product.

By definition of $P$, we have the following property: suppose that $x>a$, for some integer $a$. Then
\begin{equation}\label{eq:propertygeneral}
P^{l}_{x,y}(\raiz_{1}, \dots, \raiz_{m})=P^{l}_{x-a,y-a}(\raiz_{a+1}, \dots, \raiz_{m}).
\end{equation}
In particular, if $x>1$ then we have
\begin{equation}\label{eq:property}
P^{l}_{x,y}(\raiz_{1}, \dots, \raiz_{m})=P^{l}_{x-1,y-1}(\raiz_{2}, \dots, \raiz_{m}).
\end{equation}

The next lemma presents a recursive formula for $P$.
\begin{lem}\label{lem:recP}
Given $x,y,l$ such that $1 \leqslant x<y \leqslant m$ and $0\leqslant l < y-x$, we have that
\begin{equation*}
P^{l+1}_{x,y}(\raiz_{1}, \dots, \raiz_{m}) = \sum_{k=x+1}^{y-l-1} \CKalpha{x}{k} P^{l}_{ k, y}(\raiz_{1}, \dots, \raiz_{m})
\end{equation*}
\end{lem}
\begin{proof}
If $l=0$, it follows that
\begin{equation*}
\smashoperator{\sum_{k=x+1}^{y-1}} \CKalpha{x}{k} P^{0}_{k, y}(\raiz_{1}, \dots, \raiz_{m}) \!=\! \smashoperator{\sum_{k=x+1}^{y-1}} \CKalpha{x}{k} \CKalpha{k}{y} \!=\! \smashoperator{\sum_{x<k<y}} \CKalpha{x}{k}\CKalpha{k}{y} \!=\! P^{1}_{x,y}(\raiz_{1}, \dots, \raiz_{m}).
\end{equation*}
For $l>0$, it follows that
\begin{align*}
\smashoperator{\sum_{k=x+1}^{y-l-1}} \CKalpha{x}{k} P^{l}_{k, y}(\raiz_{1}, \dots, \raiz_{m}) & = \smashoperator{\sum_{k=x+1}^{y-l-1}} \CKalpha{x}{k} \hspace{-1em} \sum_{k<j_1<\cdots<j_l<y}\hspace{-1em} \CKalpha{k}{j_{1}} \CKalpha{j_{1}}{j_{2}} \cdots  \CKalpha{j_{l-1}}{j_{l}}\CKalpha{j_{l}}{y}\\
& = \hspace{-1em}  \sum_{x<k<j_1<\cdots<j_l<y} \hspace{-1em} \CKalpha{x}{k} \CKalpha{k}{j_{1}} \CKalpha{j_{1}}{j_{2}} \cdots  \CKalpha{j_{l-1}}{j_{l}}\CKalpha{j_{l}}{y}\\
& = P^{l+1}_{x,y}(\raiz_{1}, \dots, \raiz_{m}). \qedhere
\end{align*}
\end{proof}

It motivates the following general result. 

\begin{prop}\label{prop:raizes} Let $\raiz_{1}, \dots, \raiz_{m}$, with $m>1$, be an ordered sequence of simple roots whose simple reflections are, respectively, $s_{1}, \dots, s_{m}$. Then
\begin{equation}\label{eq:beta_n}
s_1 \cdots s_{m-1}(\raiz_{m}) =\raiz_{m} + \sum_{i=1}^{m-1} \left(\sum_{l=0}^{m-i-1} (-1)^{l-1} P^{l}_{i,m}(\raiz_{1},\dots, \raiz_{m}) \right) \cdot \raiz_{i}.
\end{equation}
\end{prop}
\begin{proof}
We will prove by induction in the number $m$ of roots. For $m=2$, we have that $s_{1}(\raiz_{2}) = \raiz_{2} - \CKalpha{1}{2}\raiz_{1}$ which satisfies Equation \eqref{eq:beta_n}.

For $m>2$, denote $s_1 \cdots s_{m-1}(\raiz_{m}) = s_{1} (\delta)$ such that $\delta = s_2 \cdots s_{m-1}(\raiz_{m})$. If we consider the ordered sequence of roots $\raiz_{1}' = \raiz_{2}, \dots, \raiz_{m-1}'= \raiz_{m}$ which has $(m-1)$ elements, it is possible to apply the inductive hypothesis in $\delta$ such that
\begin{align*}
\delta =s_2 \cdots s_{m-1}(\raiz_{m}) &=\raiz_{m-1}' + \sum_{i=1}^{(m-1)-1} \left(\sum_{l=0}^{(m-1)-i-1} (-1)^{l-1} P^{l}_{i,m-1}(\raiz_{1}',\dots, \raiz_{m-1}') \right) \cdot \raiz_{i}'\\
&= \raiz_{m} + \sum_{i=2}^{m-1} \left(\sum_{l=0}^{m-i-1} (-1)^{l-1} P^{l}_{i-1,m-1}(\raiz_{2},\dots, \raiz_{m}) \right) \cdot \raiz_{i}.
\end{align*}

For every $2\leqslant i \leqslant m-1$ and $0\leqslant l \leqslant m-i-1$, by Equation \eqref{eq:property}, we have that $P^{l}_{i-1,m-1}(\raiz_{2}, \dots, \raiz_{m}) = P^{l}_{i,m}(\raiz_{1},\raiz_{2}, \dots, \raiz_{m})$. Then,
\begin{align*}%\label{eq:prop1}
\nonumber s_1 (\delta) &=s_{1} \left( \raiz_{m} + \sum_{i=2}^{m-1} \left(\sum_{l=0}^{m-i-1} (-1)^{l-1} P^{l}_{i,m}(\raiz_{1},\raiz_{2},\dots, \raiz_{m}) \right) \cdot \raiz_{i}\right)\\
&= \raiz_{m} + \sum_{i=2}^{m-1} \left(\sum_{l=0}^{m-i-1} (-1)^{l-1} P^{l}_{i,m}(\raiz_{1},\raiz_{2},\dots, \raiz_{m}) \right) \cdot \raiz_{i}\\\nonumber & \qquad - \CK{\raiz_{1}}{\raiz_{m} + \sum_{i=2}^{m-1} \left(\sum_{l=0}^{m-i-1} (-1)^{l-1} P^{l}_{i,m}(\raiz_{1},\raiz_{2},\dots, \raiz_{m}) \right) \cdot \raiz_{i}} \cdot\, \raiz_{1}
\end{align*}

All coefficients which accompany $\raiz_{2},\dots, \raiz_{m}$ are equal to the respective coefficients in  Equation \eqref{eq:beta_n}. It remains to verify that it is also true for the coefficients that accompany the root $\raiz_{1}$ in both equations, i.e., 
\begin{equation*}%\label{eq:prop2}
\sum_{l=0}^{m-2}\! (-1)^{l-1} P^{l}_{1,m}(\raiz_{1},\dots, \raiz_{m}) \!=\! - \CK{\raiz_{1}}{\raiz_{m}} -\!\! \sum_{i=2}^{m-1} \sum_{l=0}^{m-i-1}\!\!\! (-1)^{l-1} P^{l}_{i,m}(\raiz_{1},\raiz_{2},\dots, \raiz_{m}) \cdot \CK{\raiz_{1}}{\raiz_{i}}.
\end{equation*}

We will show that the right side is equal to the left one. Firstly, by a change in the order of the summation, we have that
\begin{align*}
\sum_{i=2}^{m-1} \sum_{l=0}^{m-i-1} (-1)^{l-1} & P^{l}_{i,m}(\raiz_{1},\raiz_{2},\dots, \raiz_{m})  \cdot \CK{\raiz_{1}}{\raiz_{i}} \\
&= \sum_{l=0}^{m-3}  (-1)^{l-1} \sum_{i=2}^{m-l-1}  \CK{\raiz_{1}}{\raiz_{i}} P^{l}_{i,m}(\raiz_{1},\raiz_{2},\dots, \raiz_{m})\\
&= \sum_{l=0}^{m-3}  (-1)^{l-1} P^{l+1}_{1,m}(\raiz_{1},\dots, \raiz_{m})  \quad (\mbox{by Lemma \ref{lem:recP}})
\end{align*}

Finally, by definition, $- \CK{\raiz_{1}}{\raiz_{m}} = - P^{0}_{1,m}(\raiz_{1},\dots, \raiz_{m})$. Hence,
\begin{align*}
 - &\CK{\raiz_{1}}{\raiz_{m}} -  \sum_{i=2}^{m-1} \sum_{l=0}^{m-i-1} (-1)^{l-1} P^{l}_{i,m}(\raiz_{1},\raiz_{2},\dots, \raiz_{m})  \cdot \CK{\raiz_{1}}{\raiz_{i}} \\
& = -P^{0}_{1,m}(\raiz_{1},\dots, \raiz_{m})+  \sum_{l=0}^{m-3}  (-1)^{l} P^{l+1}_{1,m}(\raiz_{1},\dots, \raiz_{m}) = \sum_{l=0}^{m-2}  (-1)^{l-1} P^{l}_{ 1,m}(\raiz_{1},\dots, \raiz_{m}). \qedhere
\end{align*}
\end{proof}

\section{A new formula for coefficient of the boundary map}\label{sec:height}

In this section we present a formula for the coefficients of the boundary map of real flag manifolds by means of the height of some roots in the Lie algebra. Our main applications will be given alone in the context of split real forms.

\subsection{Algebraic Expression for the coefficients}
Let us begin by reviewing some main results about the determination of the cellular homology coefficients following \cite{RSm19}.
We start in the context of the maximal flag manifold $\mathbb{F}$. Given a Schubert variety $\mathcal{S}_w$, we fix once and for all reduced decompositions
\begin{equation*}
w=s_{1}\cdots s_{\ell}
\end{equation*}
as a product of simple reflections, for each $w\in \mathcal{W}$, with $\ell=\ell(w)$. Let $\mathcal{C}$ be the $\mathbb{Z}$-module freely generated by $\mathcal{S}_{w}$, $w\in \mathcal{W}$. The boundary map $\partial$ defined over $\mathcal{C}$ is given by $\partial \mathcal{S}_{w}=\sum_{w^{\prime}}c(w,w^{\prime })\mathcal{S}_{w^{\prime }}$, where $c(w,w^{\prime })\in \mathbb{Z}$ in such way that non-trivial coefficients may occur when $w$ covers $w'$. Furthermore, the non-trivial coefficients must be equal to $\pm 2$ (\cite{RSm19}, Theorem 2.2). 

Notice also that, by \cite{RSm19} Proposition 1.10, the condition for $w,w'$ to be a covering pair is equivalent to say that if $w=s_{1}\cdots s_{\ell}$ is a reduced decomposition of $w\in \mathcal{W}$ as a product of simple reflections, then $w'=s_{1}\cdots \widehat{s_{\I}}\cdots s_{\ell}$ is a uniquely defined reduced decomposition with $\mathfrak{g}(\alpha_{\I})\cong \mathfrak{sl}(2,\mathbb{R})$.

It will also be useful to denote by $v=s_{1} \cdots s_{\I-1}$ and $u=s_{\I+1}\cdots s_{\ell}$ such that $w'=v\cdot u$. There are roots not necessarily simple $\beta=v(\alpha_{\I})$ and $\gamma=u^{-1}(\alpha_{\I})$ such that
\begin{equation*}
w =s_{\beta}\cdot w' \quad \mbox {and} \quad w = w'\cdot s_{\gamma}.
\end{equation*}

Let us determine if $c(w,w')$ is either $0$ or $\pm 2$. We define
\begin{equation}\label{forsigmasoma} 
\sigma(w,w^{\prime}) =\sum_{\raiz \in \Pi _{u}}\langle \alpha_{\I}^{\vee},\raiz \rangle \cdot \dim \mathfrak{g}_{\raiz } 
\end{equation}

For $w \in \mathcal{W}$, let
\begin{equation*}%\label{forphisoma}
\phi (w)=\sum_{\raiz \in \Pi _{w}}\dim \mathfrak{g}_{\raiz}\cdot \raiz
\end{equation*}

The following results show how we determine when $c(w,w')$ is either $0$ or $\pm 2$.

\begin{prop}[\cite{RSm19}, Proposition 2.7] Let $\beta$ the unique root such that $w=s_{\beta}w'$. Then 
\begin{equation*}
\phi(w)-\phi(w')=\kappa(w,w')\cdot \beta
\end{equation*}
where $\kappa(w,w')=1-\sigma(w,w')$.
\end{prop}

\begin{thm}[\cite{RSm19}, Thm. 2.8, \cite{Koc95}, Thm. 1.1.4]
Suppose that $w$ covers $w'$. Then the coefficient $c(w,w')$ is given as follows:
\begin{equation*}
c(w,w')=\pm\left( 1+(-1)^{\kappa(w,w')} \right)
\end{equation*}
\end{thm}

Now, we address the question about the determination of the signal. This method is developed in \cite{RSm19} whose argument is based on the reduced decompositions of the elements $w\in \mathcal{W}$.

For each $w$ fix a reduced decomposition. There is a first $(-1)^\I$ ingredient which is related with the deleted position. The second component appears when the fixed reduced decomposition for $w'$ is not equal to that $\widehat{w}_\I$. According to \cite{RSm19} Proposition 1.9, there are characteristic maps for $\mathcal{S}_{w'}$ given by $\Phi_{w'}\colon B_{w'}\rightarrow \mathcal{S}_{w'}$ and $\Phi_{\widehat{w}_\I}\colon B_{\widehat{w}_\I}\rightarrow \mathcal{S}_{w'}$, where $B_{w'}$ and $B_{\widehat{w}_\I}$ are balls of dimension $\ell(w')$. The first map is obtained by the choice of reduced decomposition for $w'$ whereas the latter follows from the deletion operation. In addition, there is a property where both maps $\Phi_{w'}$ and $\Phi_{\widehat{w}_\I}$ are diffeomorphisms when restricted to the interior of the respective balls.

\begin{thm}[\cite{RSm19}, Theorem 2.8]\label{thm:rabelosanmartin}
$\displaystyle c(w,w^{\prime })=(-1)^{\I}\cdot\deg\left(\Phi_{w'}^{-1}\circ \Phi_{\widehat{w}_\I} \right)\cdot(1+(-1)^{\kappa \left( w,w^{\prime }\right)})$.
\end{thm}

\begin{rem}
When both reduced decompositions $w'$ and $\widehat{w}_\I$ are equal then $\deg\left(\Phi_{w'}^{-1}\circ \Phi_{\widehat{w}_\I} \right)=1$.

We compute the degree of the composition $\Phi_w^{-1}\circ \Phi_{\widehat{w}_\I}$ considered as map between spheres in which the boundaries of the balls are collapsed to points.
\begin{center}
$
\xymatrix{
B_{w'} \ar[r]^{\Phi_{w'}} \ar[d] & \mathcal{S}_{w'} \ar[d] & B_{\widehat{w}_{\I}} \ar[l]_{\Phi_{\widehat{w}_{\I}}}\ar[d] \\
B_{w'}/\partial(B_{w'}) \ar[r]  & \sigma_{w'} & B_{\widehat{w}_{\I}}/ \partial(B_{\widehat{w}_{\I}}) \ar[l]
}
$
\end{center}
We denote by $\sigma_{w'}=\mathcal{S}_{w'}/(\mathcal{S}_{w'}\setminus N\cdot w'b_0)$ the space obtained by identifying the complement of the Bruhat cell $N\cdot w'b_0$ to a point.
\end{rem}

\begin{rem} Let $\Theta \subset \Sigma$ and consider the partial flag manifold $\mathbb{F}_{\Theta}$. The coefficients $c(w,w')$, where $w,w' \in \mathcal{W}^{\Theta}$, are obtained in the same way as a restriction of the boundary operator to the $\mathbb{Z}$-module generated by the corresponding Schubert cells (for details, see \cite{RSm19} Theorem 3.4).
\end{rem}

\subsection{Height formula}
We now seek to find a formula for the coefficients of the cellular homology in terms of the height of the roots. 

This formula previously established for $\kappa(w,w')$ takes the relationship of $w$ and $w'$ by a left action ($w=s_{\beta}\cdot w'$) into account. We will show how to obtain an equivalent expression by exploring the right action, i.e, $w=w'\cdot s_{\gamma}$, which becomes impressively simple for split real forms.

Assume that $w=s_{1}\cdots s_{\ell}$ and $w'=s_{1}\cdots \widehat{s_{\I}} \cdots s_{\ell}=v \cdot u$, with $v=s_{1} \cdots s_{\I-1}$ and $u=s_{\I+1}\cdots s_{\ell}$, be reduced decompositions such that $(\raiz_{1},\dots, \raiz_{\ell})$ is the corresponding sequence of simple roots associated with this decomposition. Our strategy consists on finding an explicit computation of $\sigma(w,w')$ as defined in Equation \eqref{forsigmasoma} in terms of the root $u$. When $u$ is non-trivial, by Equation \eqref{eq:betat}, recall that the roots of $\Pi_{u} = \Pi^{+}\cap u\Pi^{-}$ are given by
\begin{equation*}
\beta_{j} = s_{\I+1} \cdots s_{j-1} (\raiz_{j}), \quad j \in [\I+1, \ell].
\end{equation*}

\begin{lem}\label{lem:bracket} For every $j \in [\I+1, \ell]$,
\begin{equation*}
\CK{\raiz_{\I}}{\beta_{j}} = \sum_{l=0}^{j-\I-1} (-1)^{l} P^{l}_{\I,j}(\raiz_{1},\dots, \raiz_{\ell}).
\end{equation*}
\end{lem}
\begin{proof}
If $j = \I+1$, by the definition of $P$, it follows that
\begin{equation*}
\sum_{l=0}^{j-\I-1} (-1)^{l} P^{l}_{\I,j}(\raiz_{1},\dots, \raiz_{\ell}) = (-1)^{0} P^{0}_{\I,\I+1}(\raiz_{1},\dots,\raiz_{\ell}) = \CKalpha{\I}{\I+1}.
\end{equation*}

Now, suppose that $j>\I+1$ (i.e., $\I < \ell -1$). The root $\beta_{j}$ may be written as
\begin{equation*}
\beta_{j} = s_{\I+1} \cdots s_{j-1} (\raiz_{j}) = s_{1}' \cdots s_{j-\I-1}' (\raiz_{j-\I}'),
\end{equation*}
where $s_{k}' = s_{\I+k}$ for $k\in [j-\I]$ with its respective roots. By Proposition \ref{prop:raizes}, 
\begin{align*}
\beta_{j} &= \raiz_{j-\I}' + \sum_{k=1}^{j-\I-1} \left(\sum_{l=0}^{j-\I-k-1} (-1)^{l-1} P^{l}_{k,j-i}(\raiz_{1}',\dots, \raiz_{j-\I}') \right) \cdot \raiz_{k}'\\
&= \raiz_{j} + \sum_{k=1}^{j-\I-1} \left(\sum_{l=0}^{j-\I-k-1} (-1)^{l-1} P^{l}_{k,j-\I}(\raiz_{\I+1},\dots, \raiz_{j}) \right) \cdot \raiz_{\I+k}.
\end{align*}

By Equation \eqref{eq:propertygeneral},
\begin{equation*}
P^{l}_{k,j-\I}(\raiz_{\I+1},\dots, \raiz_{j})  = P^{l}_{k+\I, j}(\raiz_{1},\dots, \raiz_{j})  = P^{l}_{k+\I, j}(\raiz_{1},\dots, \raiz_{j}, \dots, \raiz_{\ell}).
\end{equation*}

We may write the roots $\beta_j$ as  
\begin{align*}
\beta_{j}&= \raiz_{j} + \sum_{k=1}^{j-\I-1} \left(\sum_{l=0}^{j-\I-k-1} (-1)^{l-1} P^{l}_{k+\I, j}(\raiz_{1},\dots, \raiz_{\ell}) \right) \cdot \raiz_{\I+k}\\
&= \raiz_{j} + \sum_{k=\I+1}^{j-1} \left(\sum_{l=0}^{j-k-1} (-1)^{l-1} P^{l}_{k,j}(\raiz_{1},\dots, \raiz_{\ell}) \right) \cdot \raiz_{k}.
\end{align*}

Hence,
\begin{equation*}%\label{eq:propCK1}
\CK{\raiz_{\I}}{\beta_{j}} =  \CKalpha{\I}{j}+ \sum_{k=\I+1}^{j-1} \sum_{l=0}^{j-k-1} (-1)^{l-1} P^{l}_{k,j}(\raiz_{1},\dots, \raiz_{\ell})  \CKalpha{\I}{k}.
\end{equation*}

By a change in order of the summation, we have that
\begin{align*}
\CK{\raiz_{\I}}{\beta_{j}} & =  P^{0}_{\I,j}(\raiz_{1},\dots, \raiz_{\ell})+ \sum_{l=0}^{j-\I-2} \sum_{k=\I+1}^{j-l-1} (-1)^{l-1} P^{l}_{k,j}(\raiz_{1},\dots, \raiz_{\ell})  \CKalpha{\I}{k}\\
& =  P^{0}_{\I,j}(\raiz_{1},\dots, \raiz_{\ell})+ \sum_{l=0}^{j-\I-2} (-1)^{l-1} \sum_{k=\I+1}^{j-l-1} \CKalpha{\I}{k} P^{l}_{k,j}(\raiz_{1},\dots, \raiz_{\ell})\\
& =  P^{0}_{\I,j}(\raiz_{1},\dots, \raiz_{\ell})+ \sum_{l=0}^{j-\I-2} (-1)^{l-1} P^{l+1}_{\I,j}(\raiz_{1},\dots, \raiz_{\ell}) \qquad ( \mbox{by Lemma  \ref{lem:recP}})\\
& = \sum_{l=0}^{j-\I-1} (-1)^{l} P^{l}_{\I,j}(\raiz_{1},\dots, \raiz_{\ell}). \qedhere
\end{align*}
\end{proof}

\begin{cor}\label{coro:kappa}$\displaystyle\kappa(w,w') = 1+ \sum_{j=\I+1}^{\ell} \left(\sum_{l=0}^{j-\I-1} (-1)^{l-1} P^{l}_{\I,j}(\raiz_{1},\dots, \raiz_{\ell}) \right) \cdot \dim \left(\g_{\beta_{j}}\right)$.
\end{cor}
\begin{proof}
It follows as a direct application of the Lemma \ref{lem:bracket} on Equation \ref{forsigmasoma}.
\end{proof}
The expression of $\kappa$ in Corollary \ref{coro:kappa} becomes plausible when the Lie algebra is a split real form.

\begin{thm}\label{thm:dualheight}
Assume that $\g$ is a split real form. Let $\gamma=u^{-1}(\raiz_{\I})$ be the root for which $w = w'\cdot s_{\gamma}$. Then 
\begin{equation*}
\kappa(w,w')= \height(\gamma^{\vee}),
\end{equation*}
where $\height(\gamma^{\vee})$ is the height of the coroot $\gamma^{\vee}$ in the dual root system $\Pi^{*}$.
\end{thm}
\begin{proof}
By hypothesis, $\g$ is split real form, which means that $\dim\g_{\alpha}=1$, for every root $\alpha\in \Pi$. First of all, if $\I =\ell$ then $\gamma =\raiz_{\ell}$ is a simple root and $\gamma^{\vee} = \raiz_{\ell}^{\vee}$. Hence, $\kappa(w,w') = 1= \height(\gamma^{\vee})$.

Suppose that $\I < \ell$, the root $\gamma$ can be written as $\gamma = s_{\ell} s_{\ell -1}\cdots s_{\I+1}(\raiz_{\I})$. By Corollary \ref{cor:dualroot}, the coroot of $\gamma$ is $\gamma^{\vee} = s_{\raiz_{\ell}^{\vee}} s_{\raiz_{\ell -1}^{\vee}}\cdots s_{\raiz_{\I+1}^{\vee}}(\raiz_{\I}^{\vee})$. Consider the sequence of coroots $\delta_{1}',\dots, \delta_{\ell-\I+1}'$ given by $\raiz_{j}' = \raiz_{\ell -j+1}^{\vee}$, for $j \in [\ell-\I+1]$, and their simple reflections $s_{1}' = s_{\raiz_{\ell}^{\vee}}, \dots, s_{\ell-\I+1}'= s_{\raiz_{\I}^{\vee}}$. By Proposition \ref{prop:raizes},
\begin{align*}
\gamma^{\vee} = s_{1}' \dots s_{\ell-\I}' (\raiz_{\ell-\I+1}') & = \raiz_{\ell-\I+1}' + \sum_{j=1}^{\ell-\I} \left(\sum_{l=0}^{\ell-\I-j} (-1)^{l-1} P^{l}_{j,\ell-\I+1}(\raiz_{1}',\dots, \raiz_{\ell-\I+1}') \right) \cdot \raiz_{j}'\\
& = \raiz_{\I}^{\vee} + \sum_{j=1}^{\ell-\I} \left(\sum_{l=0}^{\ell-\I-j} (-1)^{l-1} P^{l}_{j,\ell-\I+1}(\raiz_{\ell}^{\vee},\dots, \raiz_{\I}^{\vee}) \right) \cdot \raiz_{\ell-j+1}^{\vee}\\
& = \raiz_{\I}^{\vee} + \sum_{j=\I+1}^{\ell} \left(\sum_{l=0}^{j-\I-1} (-1)^{l-1} P^{l}_{\ell-j+1,\ell-\I+1}(\raiz_{\ell}^{\vee},\dots, \raiz_{\I}^{\vee}) \right) \cdot \raiz_{j}^{\vee},
\end{align*}
where we have replaced $j$ by $\ell-j+1$ in the last equation. For every $j \in [\I+1,\ell]$,
\begin{align*}
P^{0}_{\ell-j+1,\ell-\I+1}(\raiz_{\ell}^{\vee},\dots, \raiz_{\I}^{\vee})& = P^{0}_{ \ell-j+1,\ell-\I+1}(\raiz_{1}',\dots, \raiz_{\ell-\I+1}') = \CK{\raiz_{\ell-j+1}'}{\raiz_{\ell-\I+1}'} = \CKalphadual{j}{\I} \\
& = \CKalpha{\I}{j}= P^{0}_{\I,j}(\raiz_{1},\dots, \raiz_{\ell})
\end{align*}
since $\raiz_{j} = \raiz_{j}^{\vee\vee}$. For $j \in [\I+1,\ell]$ and $l \in [j-\I-1]$,
\begin{align*}
P^{l}_{\ell-j+1,\ell-\I+1}&(\raiz_{\ell}^{\vee},\dots, \raiz_{\I}^{\vee}) = P^{l}_{\ell-j+1,\ell-\I+1}(\raiz_{1}',\dots, \raiz_{\ell-\I+1}')\\
&= \sum_{\ell-j+1<j_1<\cdots<j_l<\ell-\I+1} \CK{\raiz_{\ell-j+1}'}{\raiz_{j_{1}}'} \CK{\raiz_{j_{1}}'}{\raiz_{j_{2}}'} \cdots\CK{\raiz_{j_{l}}'}{\raiz_{\ell-\I+1}'}\\
&= \sum_{\ell-j+1<j_1<\cdots<j_l<\ell-\I+1} \CKalphadual{j}{\ell-j_{1}+1}\cdots\CKalphadual{\ell-j_{l}+1}{\I}\\
&= \sum_{j>\ell-j_1+1>\cdots>\ell-j_l+1>\I} \CKalpha{\ell-j_{1}+1}{j}\cdots\CKalpha{\I}{\ell-j_{l}+1}\\
%&= \sum_{\I<k_{1}< \cdots < k_{l}<j} \CKalpha{j}{k_{\ell}} \cdots \CKalpha{k_{1}}{\I} 
&= \sum_{\I<k_{1}< \cdots < k_{l}<j} \CKalpha{\I}{k_{1}} \cdots \CKalpha{k_{l}}{j} = P^{l}_{\I,j}(\raiz_{1},\dots,\raiz_{\ell}).
\end{align*}

Hence,
\begin{equation*}
\gamma^{\vee} = \raiz_{\I}^{\vee} + \sum_{j=\I+1}^{\ell} \left(\sum_{l=0}^{j-\I-1} (-1)^{l-1} P^{l}_{\I,j}(\raiz_{\ell},\dots, \raiz_{\I}) \right) \cdot \raiz_{j}^{\vee}.
\end{equation*}

Therefore, by Corollary \ref{coro:kappa}, the height of the coroot $\gamma^{\vee}$ is
\begin{equation*}
\height(\gamma^{\vee}) = 1 + \sum_{j=\I+1}^{\ell} \left(\sum_{l=0}^{j-\I-1} (-1)^{l-1} P^{l}_{\I,j}(\raiz_{\ell},\dots, \raiz_{\I}) \right) = \kappa(w,w').  \qedhere
\end{equation*}
\end{proof}

\begin{cor}\label{coro:coefcases}
Let $\g$ be a Lie algebra of type $A_{n}, D_{n}, E_{6}, E_{7}$ or $ E_{8}$. Then
\begin{equation*}
\kappa(w,w')  = \height(\gamma) 
\end{equation*}
\end{cor}
\begin{proof}
In the context of a split real form of a Lie algebra of type $A_{n}, D_{n}, E_{6}, E_{7}$ or $ E_{8}$, all roots have the same length. Then, $\mathrm{ht}(\alpha)=\mathrm{ht}(\alpha^{\vee})$ for all $\alpha \in \Pi$.
\end{proof}

\begin{rem}
We can use Theorem \ref{thm:dualheight} to get the formula for $F_{4}$ and $G_{2}$ Lie algebras. If $\g$ is a Lie algebra of type $F_{4}$, suppose that the simple roots are ordered canonically as follows 
\begin{tikzpicture}
\dynkin[mark=o]{F}{4};\dynkinLabelRoot*{1}{a_1} \dynkinLabelRoot*{2}{a_2} \dynkinLabelRoot*{3}{a_3} \dynkinLabelRoot*{4}{a_4}
\end{tikzpicture}. Then,
\begin{equation*}
\kappa(w,w')  = \height(\widetilde\gamma)
\end{equation*}
where $\widetilde\gamma$ is the root obtained from $\gamma$ by reversing the simple roots $\Sigma = \{a_{1}, a_{2}, a_{3}, a_{4}\}$ to $\Sigma^{\vee} = \{a_{4}, a_{3}, a_{2}, a_{1}\}$, respectively, since the dual simple roots correspond to reverse roots in the Dynkin diagram, i.e., $a_{1}^{\vee} = a_{4}, a_{2}^{\vee} = a_{3}, a_{3}^{\vee} = a_{2}, a_{4}^{\vee} = a_{1}$. The same idea applies to the Lie algebra $G_{2}$.
\end{rem}

\begin{rem}
We also can use Theorem \ref{thm:dualheight} to get the formula for $B_{n}$ and $C_{n}$ Lie algebras from the isomorphism $\Pi_{B}^{*} \cong \Pi_{C}$. For instance, in the context of isotropic and odd orthogonal Grassmannians, it coincides with Theorem 3.12 of \cite{LR19a}.
\end{rem}

\subsection{Type A case}

In this section, we present some immediate consequences of Theorem \ref{thm:dualheight} for flags associated with type A Lie algebras. It emphasizes the convenience of the permutation model the symmetric group provides parametrizing the Schubert cells.

Let $G=\mathrm{Sl}(n,\mathbb{R})$ be a Lie group of type $A$ and $\Sigma=\{a_{1},\dots, a_{n-1}\}$ the simple roots ordered as it follows:
\begin{center}
\begin{tikzpicture}[scale=2.5] \dynkin[mark=o]{A}{};
\dynkinLabelRoot*{1}{a_1} \dynkinLabelRoot*{2}{a_2} \dynkinLabelRoot*{4}{a_{n-1}}
\end{tikzpicture}
\end{center}
The respective Weyl group $\mathcal{W}$ is the symmetric group $S_n$. 
We denote a permutation $w\in S_{n}$ in the one-line notation by $w=w_1w_2\cdots w_\ell$ where $w_i=w(i)$ for all $i=1,\ldots, \ell$. 

The following lemma provides a characterization of the covering relation of two permutations using the one-line notation.

\begin{lem}[\cite{BB05}, Lemma 2.1.4]\label{lem:bjorner}
Let $w,w'\in S_{n}$. Then, $w$ covers $w'$ in the Bruhat order if and only if $w = w' \cdot (i,j)$ for some transposition $(i,j)$ with $i<j$ such that $w'(i) < w'(j)$ and there does not exist any $k$ such that $i<k<j$, $w'(i) < w'(k) <w'(j)$.
\end{lem}

The lemma says that if $w=w_{1}\cdots w_{n}$ is the one-line notation of $w\in S_{n}$ then $w'$ is covered by $w$ if and only if the one-line notation of $w'$ is obtained from $w$ by switching the values in position $i$ and $j$, for some pair $i<j$ and such that no value between positions $i$ and $j$ lies in $[w(j),w(i)]$.

The next proposition follows from the covering relation stated in Lemma \ref{lem:bjorner}.
\begin{prop}\label{prop:typeAcoef}
Let $w,w'\in S_{n}$ such that $w'$ is covered by $w$, i.e., $w = w' \cdot (i,j)$ for some $i<j$. Then, the coefficient $\kappa(w,w')$ is given by
\begin{equation*}
\kappa(w,w') = j-i.
\end{equation*}
In particular, $c(w,w')=0$ if, and only if, $j-i$ is odd.
\end{prop}
\begin{proof}
Since $\gamma$ is the root such that $w = w' s_{\gamma}$, the transposition $(i,j)$ is the reflection $s_{\gamma}$ through $\gamma$. Considering the simple reflections $s_{i} = s_{a_{i}}$, for $a_{i}\in \Sigma$, a reduced decomposition for $(i,j)$ is $s_{j-1} \cdots s_{i+1} s_{i} s_{i+1} \cdots s_{j-1}$. Using the fact that $s_{w(\alpha)} = w s_{\alpha} w^{-1}$ we have $(i,j) = s_{s_{j-1} \cdots s_{i+1} (a_{i})}$ and, then, $\gamma = s_{j-1} \cdots s_{i+1} (a_{i})$. Applying Proposition \ref{prop:raizes}, the root $\gamma$ is the sum of simple roots $a_{i}+ a_{i+1}+\cdots + a_{j-1}$. Therefore, by Corollary \ref{coro:coefcases}, $\kappa(w,w') = \height(\gamma) = j-i$.
\end{proof}

This proposition simplifies a lot the task to compute the boundary coefficient. For instance, given $w = 1\,3\, 7 \,5\,8\,2\,9\,4\,6$ and $w' = 1\,3\, 7 \,\mathbf{2}\,8\,\mathbf{5}\,9\,4\,6$, we see that $w = w'\cdot (4,6)$. Hence, $\kappa(w,w') = 4-2 =2$ and $c(w,w') = \pm (1 + (-1)^{\kappa(w,w')}) = \pm 2$. 

Let us give a combinatorial proof of the condition for the orientability of any real flag manifold of type A (\cite{PSS12}, Proposition 4.1) as a direct application of Proposition \ref{prop:typeAcoef}. 

\begin{prop}\label{prop:orient}
Let the complement of $\Theta$ with respect to $\Sigma$ be the set of roots $\{a_{d_{1}},a_{d_{2}},\dots, a_{d_{k}}\}$ where $d_{0} = 0<d_1<\cdots<d_k<n = d_{k+1}$. Then, the flag manifold $\mathbb{F}_{\Theta}$ is orientable if and only if $d_{j+1}-d_{j}$ have the same parity, for every $j\in [k]$.
\end{prop}
\begin{proof} We will establish a criterion for orientability by seeking in which condition the top dimensional homology group is $\mathbb{Z}$. 
Let us begin with the (unique) Schubert top cell $\schub_{w_{\Theta}}$. The associated permutation $w_{\Theta}$ is the longest permutation (with respect to the Bruhat order) with descents at positions $d_1, \ldots, d_k$. In one-line notation, 
\begin{align*}
w_{\Theta} = (d_k+1) \cdots n | (d_{k-1}+1) \cdots d_k  | \cdots  | 1 \cdots d_1
\end{align*}

There are $k$ Schubert cells $\schub_{(w_{\Theta})'_j}$, $j\in [k]$, covered by $\schub_{w_{\Theta}}$. For each $j\in [k]$, the corresponding permutations of $w_{\Theta}$ and $(w_{\Theta})'_j$ differ only by the values at positions $d_{j-1}+1$ and $d_{j+1}$, where we consider $d_0=0$ and $d_{k+1}=n$, i.e., 
\begin{align*}
w_{\Theta} = (w_{\Theta})'_j \cdot (d_{j-1}+1 , d_{j+1}).
\end{align*}

By Proposition \ref{prop:typeAcoef}, it follows that 
\begin{align*}
\kappa(w_{\Theta},(w_{\Theta})'_j)= d_{j+1} - (d_{j-1}+1),
\end{align*}
i.e., $c(w_{\Theta},(w_{\Theta})'_j)=0$ if and only if $(d_{j+1} \!-\! d_{j-1})$ is even. Since $d_{j+1} \!-\! d_{j-1} = (d_{j+1} \!-\! d_{j})\!+\!(d_{j} \!-\! d_{j-1})$, $\partial \schub_{w_{\Theta}}\!=\!0$ if and only if $d_{j+1}\!-\!d_{j}$ have the same parity, for every $j$. \qedhere
\end{proof}

\subsection{Low-dimensional integral homology}
We now seek a formula for the 1,2-homology for partial flag manifolds of type A. This require to introduce a few combinatorial notations.

The code (also called Lehmer code) of a permutation $w\in S_{n}$ is an integer sequence $\alpha$ with $\alpha_{i} = \#\{k>i \ |\ w_{k}<w_{i}\}$ and it will be denoted by $\code(w)$. In other word, each entry of the code corresponds to the number of inversions to the right of $w_{i}$. It's clear that $0\leqslant \alpha_{i}\leqslant n-i$. The code provides a bijection between $S_{n}$ and the set $[0,n-1]\times [0,n-2] \times \cdots \times [0,1]$.

Given $w\in S_{n}$, the \emph{code spectrum} of $w$ is the unique partition $0<b_{1}\leqslant b_{2}\leqslant \cdots \leqslant b_{l}<n$ such that the code $\alpha$ of $w$ is given by $\alpha_{i} = \#\{j \colon b_{j} = i\}$.  We will denote by $\w{b_{1},\cdots, b_{\ell}}$ the permutation $w$ given by such code spectrum to distinguish it from the other notations.
For instance, $w\in S_{5}$ such that $\code(w) = (0,2,0,1)$ then its code spectrum is $(2,2,4)$, i.e, $w = \w{2,2,4}$.

This notation allows us to easily describe permutations with some choose length. Let us describe all permutations $w\in S_{n}$ with length up to three:
\begin{itemize}
\item $s_{i} = \w{i}$ for $i\in[n-1]$;
\item $s_{i}s_{j}= \w{i,j}$ for $i<j$ and $i,j\in [n-1]$;
\item $s_{i+1}s_{i} = \w{i,i}$ for $i\in [n-2]$;
\item $s_{i}s_{j} s_{k} =\w{i,j,k}$ for $i < j < k$ and $i,j,k\in [n-1]$;
\item $s_{i+1}s_{i} s_{k} = \w{i,i,k}$ for $i<k$ and $i,k\in [n-1]$;
\item $s_{i}s_{j+1} s_{j} =\w{i,j,j}$ for $i < j$ and $i,j\in [n-2]$;
\item $s_{i+2}s_{i+1} s_{i} = \w{i,i,i}$ for $i \in [n-3]$.
\end{itemize}

When we require to fix a reduced decomposition, we will choose the ones as above.

The following lemma provides us all the boundary maps required to compute 1,2-homology of any partial flag manifold of type A.

\begin{lem}\label{lem:boundarylow}\
\begin{enumerate}
\item $\partial \schub_{\w{i}} = 0$, for $i\in [n-2]$;
\item $\partial \schub_{\w{i,i}} = -2\schub_{\w{i}}$, for $i\in [n-2]$;
\item $\partial \schub_{\w{i,i+1}} = -2\schub_{\w{i+1}} $, for $i\in [n-2]$;
\item $\partial \schub_{\w{i,j}} = 0$, for $i\in [n-3]$ and $j\in [i+2,n-1]$.
\item $\partial \schub_{\w{i,j-1,j}} = 2 \schub_{\w{i,j}}$, for $i\in [n-3]$ and $j\in [i+2,n-1]$;
\item $\partial \schub_{\w{i,i+1,i+1}} = 2 \schub_{\w{i,i+1}} - 2\schub_{\w{i+1,i+1}}$ for $ i\in [n-3]$;
\end{enumerate}
\end{lem}
\begin{proof}
Proposition \ref{prop:typeAcoef} gives us when the coefficient is $\pm 2$.
Recall that $\hat{w}_{\I}$ is the reduced decomposition of $w'$ obtained by removing the $I$-th simple reflection of $w$.
To get the sign we can observe that if we choose to fix the reduced decomposition for $w$ and $w'$ as given above, both reduced decompositions $w'$ and $\hat{w}_{\I}$ are exactly the same. Then, $\Phi_{w'} = \Phi_{\hat{w}_{\I}}$ and $\deg\left(\Phi_{w'}^{-1}\circ \Phi_{\widehat{w}_\I} \right) = 1$. Then, the sign of the coefficient is given by $(-1)^{\I}$.
Therefore,
\begin{enumerate}
\item For $i\in [n-1]$, $c(\w{i},e) = 0$; %Then, $\partial \schub_{\w{i}} = 0$;
\item For $i\in [n-2]$, $c(\w{i,i},\w{i}) = -2$ and $c(\w{i,i},\w{i+1}) = 0$;
\item For $i\in [n-2]$, $c(\w{i,i+1},\w{i+1}) = -2$ and $c(\w{i,i+1},\w{i}) = 0$;
\item For $i\in [n-3]$ and $j\in [i+2,n-1]$, $c(\w{i,j},\w{i}) = 0$ and $c(\w{i,j},\w{j}) = 0$;
\item For $i\in [n-3]$ and $j\in [i+2,n-1]$, $c(\w{i,j-1,j}, \w{j-1,j})=0$, $c(\w{i,j-1,j}, \w{i,j})=2$, and $c(\w{i,j-1,j}, \w{i,j-1})=0$;
\item For $i\in [n-3]$, $c(\w{i,i+1,i+1}, \w{i+1,i+1})=-2$, $c(\w{i,i+1,i+1}, \w{i,i+1})=2$, and $c(\w{i,i+1,i+1}, \w{i,i+2})=0$.
\end{enumerate}
\end{proof}

Denote by $r_{\Theta}$ the number of connected components of the Dynkin diagram of $\Theta$.

\begin{thm} Consider a partial flag manifold $\mathbb{F}_{\Theta}$ of type $A$, where $\Theta\subset\Sigma$.
\begin{enumerate}
\item For $n\geqslant 3$, the 1-homology is given by
\begin{equation*}
H_{1}(\mathbb{F}_{\Theta},\Z) \cong (\Z_{2})^{n-|\Theta|-1}.
\end{equation*}

\item For $n\geqslant 4$, the 2-homology of is given by
\begin{equation*}
H_{2}(\mathbb{F}_{\Theta},\Z) \cong (\Z_{2})^{N_{\Theta}}
\end{equation*}
where $N_{\Theta} = {n-|\Theta|-1 \choose 2} +r_{\Theta} - 1.$
\end{enumerate}
\end{thm}
\begin{proof}
To compute the 1-homology, suppose that $\Sigma-\Theta=\{a_{d_{1}},a_{d_{2}},\dots, a_{d_{k}}\}$ where $0<d_1<\cdots<d_k<n$.
All permutations of length 1 in $\weyl^{\Theta}$ are in the form $\w{d_{i}}$. By Lemma \ref{lem:boundarylow}, the kernel $\ker(\partial_{\Theta})$ is generated by $\schub_{\w{d_{i}}}$, $i\in [k]$.

For every $i\in [k]$ such that $d_{i}<n-1$, we have that $\partial \schub_{\w{d_{i},d_{i}}} = -2\schub_{\w{d_{i}}}$. The only exception is when $d_{k}=n-1$. In this case, $\partial \schub_{\w{d_{k}-1,d_{k}}} = -2\schub_{\w{d_{k}}}$. We conclude that  $H_{1}(\mathbb{F}_{\Theta},\Z)$ has no free part and the set $\{\w{d_{i}} \colon i\in [k]\}$ generates the torsion, i.e, $H_{1}(\mathbb{F}_{\Theta},\Z) \cong (\Z_{2})^k$.

To compute the 2-homology, let us prove that it has no free part, i.e., $H_{2}(\mathbb{F}_{\Theta},\Z) \cong (\Z_{2})^{x}$ for some integer $x$. Consider the maximal flag manifold $\mathbb{F}$.

By Lemma \ref{lem:boundarylow}, the kernel $\ker(\partial)$ is generated by 
\begin{itemize}
\item $X_{i,j} = \schub_{\w{i,j}}$, for $i\in [n-3]$ and $j\in [i+2,n-1]$;
\item $X_{i,i+1} = \schub_{\w{i,i+1}} - \schub_{\w{i+1,i+1}}$, for $i\in [n-3]$.
\end{itemize}

Notice that we do not allow $X_{n-2,n-2}$ since $\w{n-1,n-1}\not\in S_{n}$. Also by Lemma \ref{lem:boundarylow}, we have that each $X_{i,j}$ is image through $\partial$ of the following Schubert cells:
\begin{alignat*}{2}
\partial \schub_{\w{i,j-1,j}} &= 2 X_{i,j} && \mbox{, for } i\in [n-3] \mbox{ and } j\in [i+2,n-1],\\
\partial \schub_{\w{i,i+1,i+1}}  & = 2X_{i,i+1}  && \mbox{, for } i\in [n-3].
\end{alignat*}

These generators $X_{i,j}$ are indexed by the set $\mathcal{I}_{2}$ of pairs $(i,j)$ given by
\begin{equation*}
\mathcal{I}_{2} = \{(i,j) \colon i\in [n-3] \mbox { and } j\in [i+1,n-1]\} 
= \{(i,j)\in [n-1]^{2}\colon i<j\}-\{(n-2,n-1)\}.
\end{equation*}

Hence, $H_{2}(\mathbb{F},\Z) \cong (\Z_{2})^{|\mathcal{I}_{2}|} = (\Z_{2})^{{n-1\choose 2}-1}$ since $\mathcal{I}_{2}$ counts the number of 2-combinations in $n-1$ elements but one.

For the partial flag manifold $\mathbb{F}_{\Theta}$, consider the following cases:
\begin{itemize}
\item if $\w{i,j}\in \weyl^{\Theta}$ such that $i\in [n-3]$ and $j\in [i+2,n-1]$ then $\w{i,j-1,j} \in \weyl^{\Theta}$. Hence, $X_{i,j}$ is a generator of the torsion of $H_{2}(\mathbb{F}_{\Theta},\Z)$;
\item if $\w{i,i+1}\in \weyl^{\Theta}$ for $i\in [n-3]$ then $\w{i+1,i+1}\in \weyl^{\Theta}$ and $\w{i,i+1,i+1} \in \weyl^{\Theta}$. Hence, $X_{i,i+1}$ is a generator of the torsion of $H_{2}(\mathbb{F}_{\Theta},\Z)$.
\end{itemize}

Therefore, the 2-homology has no free part and the torsion is generated by all $X_{i,j}$ such that $\w{i,j} \in \weyl^{\Theta}$, i.e., $H_{2}(\mathbb{F}_{\Theta},\Z) \cong (\Z_{2})^{|\{ (i,j)\in\mathcal{I}_{2} \colon \w{i,j} \in \weyl^{\Theta}\}|}$.

Finally, we will prove that $|\{ (i,j)\in\mathcal{I}_{2} \colon \w{i,j} \in \weyl^{\Theta}\}| = N_{\Theta}$ by induction on the cardinality of $\Theta$. If $|\Theta| = 0$ then $|\mathcal{I}_{2}| = {n-1 \choose 2} -1 = N_{\emptyset}$. Suppose, by induction, that $|\{ (i,j)\in\mathcal{I}_{2} \colon \w{i,j} \in \weyl^{\Delta}\}| = N_{\Delta}$ for any strict subset $\Delta$ of $ \Theta$.

Assume that $\Delta = \Theta - \{a_{k}\}$, for $k\in [n-1]$, where $a_{k}$ is the greatest simple root (with respect to $\Sigma$) in $\Theta$. Since $\weyl^{\Theta}\subset\weyl^{\Delta}$, by induction
\begin{equation}\label{eq:XijTheta}
|\{ (i,j)\in\mathcal{I}_{2} \colon \w{i,j} \in \weyl^{\Theta}\}| = N_{\Delta} -|\{ (i,j)\in\mathcal{I}_{2} \colon \w{i,j} \in \weyl^{\Delta}\backslash \weyl^{\Theta}\}|
\end{equation}

Let us compute $|\{ (i,j)\in\mathcal{I}_{2} \colon \w{i,j} \in \weyl^{\Delta}\backslash \weyl^{\Theta}\}|$.
Given $(i,j)\in \mathcal{I}_{2}$ such that $\w{i,j} \in \weyl^{\Delta}$, consider the following cases:
\begin{itemize}
\item if $i\neq k$ and $j\neq k$ then $\w{i,j}\in \weyl^{\Theta}$;
\item if $i = k$ and $j = k+1$ then $\w{i,j}\in \weyl^{\Theta}$;
\item if $i = k$ and $j > k+1$ then $\w{i,j}\not\in \weyl^{\Theta}$;
\item if $j = k$ then $\w{i,j}\not\in \weyl^{\Theta}$.
\end{itemize}

Thus,
\begin{align*}
|\{ (i,j)\in\mathcal{I}_{2} \colon \w{i,j} \in \weyl^{\Delta}\backslash \weyl^{\Theta}\}| =
& |\{ (i,j) \in\mathcal{I}_{2} \colon i = k, j>k+1, \mbox{ and } \w{i,j} \in \weyl^{\Delta}\}|+\\
&+|\{ (i,j)\in\mathcal{I}_{2} \colon j = k \mbox{ and } \w{i,j} \in \weyl^{\Delta} \}|.
\end{align*}

Since $a_{k}$ is the greatest root of $\Theta$ then 
\begin{align*}
|\{ (i,j) \in\mathcal{I}_{2} \colon i = k, j>k+1, \mbox{ and } \w{i,j} \in \weyl^{\Delta}\}| & = |\{j\colon j\in [k+2,n-1]\}|\\
& =\left\{
\begin{array}{cc}
n-k-2 & \mbox{ if } k < n-1, \\ 
0 & \mbox{ if } k = n-1.
\end{array} 
\right.
\end{align*}

On the other hand,
\begin{align*}
|\{ (i,j)\in\mathcal{I}_{2} & \colon j =  k \mbox{ and } \w{i,j} \in \weyl^{\Delta} \}| =  \\
& = |\{i \colon i\in [k-1]\cap [n-3] \mbox{ and } a_{i}\not\in\Delta\}| +
\left\{
\begin{array}{cc}
1 & \mbox{ if } a_{k-1}\in \Delta, \\ 
0 & \mbox{ otherwise.}
\end{array} 
\right. \\
& = |\{i \colon i\in [k-1]\cap [n-3] \mbox{ and } a_{i}\not\in\Delta\}| + (1-(r_{\Theta} - r_{\Delta}))  \\
& = 1-(r_{\Theta} - r_{\Delta}) + 
\left\{
\begin{array}{cc}
k - 1 - |\Delta| & \mbox{ if } k<n-1, \\ 
n-3 -|\Delta| & \mbox{ if } k=n-1.
\end{array} 
\right.\\
&= k -|\Delta|-r_{\Theta}+r_{\Delta} +
\left\{
\begin{array}{cc}
0 & \mbox{ if } k<n-1, \\ 
-1 & \mbox{ if } k=n-1.
\end{array} 
\right.
\end{align*}
 
Hence, $|\{ (i,j)\in\mathcal{I}_{2} \colon \w{i,j} \in \weyl^{\Delta}\backslash \weyl^{\Theta}\}| = n-|\Delta|-2-r_{\Theta}+r_{\Delta}$ 
and, by Equation \eqref{eq:XijTheta},
\begin{align*}
|\{ (i,j)&\in \mathcal{I}_{2} \colon \w{i,j} \in \weyl^{\Theta}\}|   = {n-|\Delta|-1 \choose 2} +r_{\Delta} - 1 - (n-|\Delta|-2-r_{\Theta}+r_{\Delta})\\
%& = \dfrac{(n-|\Delta|-1)(n-|\Delta|-2)}{2} - (n-|\Delta|-2) +r_{\Theta}-1\\
& = \dfrac{(n-|\Delta|-2)(n-|\Delta|-3)}{2} +r_{\Theta}-1  = {n-(|\Delta|+1)-1 \choose 2} +r_{\Theta} - 1 = N_{\Theta}.\qedhere
\end{align*}
\end{proof}

\section{Final comments and further directions}\label{sec:conclusion}
We would like to highlight that, although this is a classical theme -- topology of real flag manifolds such as Projective spaces and Grassmannian manifolds -- we have not found in the literature such simple formula to compute its homology groups. Its remarkable how simple are the formulas for 1,2-homology of partial flag manifolds of type A.

With the results obtained in this paper, we are able to visualize other directions to research as listed below. 
\begin{itemize}
\item For type A flag manifolds, it seems possible to get a formula for 3, 4-homology. It will require to get a better understanding of the coefficient since the degree in Theorem \ref{thm:rabelosanmartin} are not as easy to compute. A forthcoming paper will deal with the combinatorics involved to explicitly compute the sign in the type A case.
\item Theorem \ref{thm:dualheight} provides a formula of the boundary coefficient for split real forms. It is reasonable to ask about low dimensional homology of other types of flag manifolds. This would require to get a nicer combinatorial model for the Weyl group.
\end{itemize}

\subsubsection*{Acknowledgments}
We thank to San Martin and Lucas Seco for helpful suggestions and valuable discussions.
This research was motivate by computer investigation using the open-source mathematical software Sage \cite{Sage}.

\bibliographystyle{plain}
\bibliography{biblio}

\end{document}